\documentclass{article}
\usepackage[utf8]{inputenc}
\usepackage{todonotes}
\usepackage{caption}
\usepackage{nicefrac}
\usepackage{amsmath,enumerate}
 \usepackage{amsthm}
\usepackage{amsfonts}
\usepackage{amssymb}
\usepackage{mathrsfs}
\usepackage{pgfplots}
\usepackage{mathtools}
\usepackage{comment}
\usepackage{lmodern}
\usepackage{bm}
\usepackage{subfig}
\usepackage{caption}
\usepackage{cite}
\usepackage[margin=1.5in]{geometry}
\usepackage{tikz}
\usetikzlibrary{shapes,arrows,graphs,graphs.standard,shapes.misc}
\usetikzlibrary{matrix,decorations.pathreplacing, calc, positioning,fit}
\usetikzlibrary{shapes.geometric}

\usepackage{tikz}
\usetikzlibrary{external}
\usetikzlibrary{decorations.shapes,arrows,calc,decorations.markings,shapes,decorations.pathreplacing,shapes.arrows}
\usetikzlibrary{positioning}
\tikzset{main node/.style={circle,fill=blue!20,draw,inner sep=1pt},}

\newtheorem{theorem}{Theorem}[section]
\newtheorem{definition}[theorem]{Definition}
\newtheorem{assumption}[theorem]{Assumption}
\newtheorem{lemma}[theorem]{Lemma}

\newtheorem{remark}[theorem]{Remark}
\newtheorem{example}[theorem]{Example}
\newtheorem{corollary}[theorem]{Corollary}
\newtheorem{proposition}[theorem]{Proposition}  

\newcommand{\subqed}{$\blacksquare$}
\newcommand{\real}{{\mathbb{R}}}
\newcommand{\N}{{\mathbb{N}}}

\newcommand{\argmin}{\operatorname{argmin}}

\newcommand{\E}{\mathcal{E}}

\definecolor{BBlue}{cmyk}{.98,0.10,0,.25}

\newcommand{\G}{\mathcal{G}}
\newcommand{\V}{\mathbb{V}}

\newcommand{\T}{\mathcal{T}}

\title{Structural averaged controllability \\ of  linear ensemble systems}

\author{Bahman Gharesifard\footnote{Department of Mathematics and Statistics at Queen's University, Kingston, ON, Canada. Email: \texttt{bahman.gharesifard@queensu.ca}.} \quad\quad
Xudong Chen\footnote{ECEE Department, CU Boulder. Email: \texttt{xudong.chen@colorado.edu}.}}

\begin{document}

\date{}
\maketitle 

\begin{abstract}
In the paper, we introduce and address the problem of structural averaged controllability for linear ensemble systems. We provide examples highlighting the differences between this problem and others. In particular, we show that structural averaged controllability is strictly weaker than structural controllability for single (or ensembles of) linear systems. We establish a set of necessary or sufficient conditions for sparsity patterns to be structurally averaged controllable.
\end{abstract}

\section{Introduction}

Consider a linear ensemble control system over the parameterization space $\Sigma:=[0,1]$: 
\begin{equation}\label{eq:ensemble}
\dot{x}(t,\sigma):=\frac{\partial}{\partial t} x(t,\sigma)=A(\sigma) x(t,\sigma)+B(\sigma)u(t),
\end{equation}
where $\sigma \in \Sigma$, $ A\in \mathrm{C}^0(\Sigma,\real^{n\times n}) $, and $ B \in \mathrm{C}^0(\Sigma,\real^{n\times m}) $, with $n$ and $m$ positive integers, $ u(t) \in\real^m $ is the control input at time $ t\geq 0 $, and $ x(t,\sigma) \in \real^n $ is the state of the individual system indexed by $\sigma$ at time $t$. 

We review a few known controllability results associated with system~\eqref{eq:ensemble}. 
First, each individual system is a linear time-invariant system, and is controllable if and only if the columns of $\{A^k(\sigma) B(\sigma)\}^{n-1}_{k = 0}$ span $\real^n$.  

Next, we say that the ensemble system~\eqref{eq:ensemble} is $\mathrm{L}^p$-controllable, for $1\le p\le \infty$, if for any initial profile of the ensemble $ x(0,\cdot) \in \mathrm{C}^0(\Sigma,\real^n) $, any target profile $ x^* \in \mathrm{C}^0(\Sigma,\real^n) $, any error tolerance $\epsilon > 0$, and any time $T > 0$, there exists a control input $ u \in \mathrm{L}^1([0,T];\real^m) $ such that 
$\| x(T,\cdot) -x^*\|_{\mathrm{L}^p}< \epsilon 
$.  When $p = \infty$, $\mathrm{L}^\infty$-controllability is also commonly referred to as uniform controllability. 

It is known~\cite{RT:75} that the ensemble system~\eqref{eq:ensemble} is $\mathrm{L}^p$-controllable if the $ \mathrm{L}^p $-closure of the vector space spanned by the columns of $\{A^kB\}_{k=0}^{\infty}$ is $\mathrm{L}^p(\Sigma,\real^n)$ for $p < \infty$ and $\mathrm{C}^0(\Sigma,\real^n)$ for $p = \infty$. We further refer the reader to~\cite{li2015ensemble,helmke2014uniform,chen2020controllability} and references thereafter for other controllability results for linear ensemble systems. 

Finally, we recall the notion of {\em averaged controllability}~\cite{EZ:14}.
The ensemble system~\eqref{eq:ensemble} is said to be {\em averaged controllable} if for any initial state $ x_0\in \real^n$, with $ \int_0^1x(0,\sigma)=x_0 $, any target state $ x^*\in \real^n$, and any time $ T>0 $, there is a control input $ u \in \mathrm{L}^1([0,T];\real^m) $ such that 
$
\int_0^1x(T,\sigma)\mathrm{d}\sigma=x^*
$. 
The following result~\cite[Theorem~1]{EZ:14} presents a necessary and sufficient condition on averaged controllability:  

\begin{theorem}\label{theorem:averaged}
The ensemble control system~\eqref{eq:ensemble} is averaged controllable if and only if the vector space spanned by the columns of $ \{\int_0^1A^j(\sigma)B(\sigma)\mathrm{d}\sigma\}_{j\geq 0} $ is of rank $ n $. 
\end{theorem}

In this paper, we will consider sparse matrix pairs $(A, B)$ and introduce a novel structural controllability problem, namely, the {\em structural averaged controllability problem}.

The sparsity pattern of $(A, B)$,  i.e., the locations of non-zero entries, can be represented by a directed graph (digraph) $\G = (V, \E)$  on $(n + m)$ nodes, with the node set given by $V =\{\alpha_1,\ldots \alpha_n,\beta_1,\ldots, \beta_m\}$. The $ \alpha $-nodes correspond to the states and the $\beta$-nodes correspond to the control inputs. We denote by $V_\alpha:= \{\alpha_1,\ldots, \alpha_n\}$ and $V_\beta := \{\beta_1,\ldots, \beta_n\}$.   
The edge set $ \E $ is defined as follows:
\begin{enumerate}
\item There is no incoming edge to any $ \beta $-node;
\item If the $ji$th entry of $A$ is not  a zero function, then $(\alpha_i,\alpha_j) \in \E $; 
\item Similarly, if the $ji$th entry of $B$ is not a zero function, then $(\beta_i,\alpha_j) \in \E$.  
\end{enumerate}
Conversely, to any such digraph $\G$, one can assign a class of matrix pairs $(A', B')$ of appropriate size (i.e., $A$ is $n\times n$ and $B$ is $n\times m$) such that their sparsity patterns $\G'$ are subgraphs of $\G$. All of these pairs $(A',B')$ form a vector space, which we denote by $\V(\G)$. Any pair $(A',B')$ in the space $\V(\G)$ is said to be {\em compliant with} $\G$.  

In the remainder of the section, we first recall the classical structural controllability problem introduced by Lin~\cite{CTL:74} and the structural {\em ensemble} controllability problem introduced in~\cite{XC:20}. After that, we introduce the problem of structural averaged controllability that will be investigated in this paper.

\subsection{Structural controllability}
Consider finite-dimensional linear time-invariant systems 
\begin{equation}\label{eq:singlelinearsystem}
\dot x(t) = A x(t) + Bu(t).
 \end{equation}
 Here, pairs $(A, B)$ are simply constant matrices (instead of being matrix-valued functions).  
Correspondingly, for a given sparsity pattern $\G$, the vector space $\V(\G)$ is now a subspace of $\real^{n\times n} \times \real^{n\times m}$. 
A sparsity pattern $\G$ is said to be structural controllable if there exists a pair $(A, B)\in \V(\G)$ such that the resulting linear system~\eqref{eq:singlelinearsystem} is controllable. 
Necessary and sufficient conditions for $\G$ to be structurally controllable have been derived and presented in various forms~\cite{CTL:74,shields1976structural,glover1976characterization,olshevsky2015minimum}. We provide below a graphical condition. %

We say that the digraph $\G = (V, \E)$ is {\em accessible} to the $\beta$-nodes if for each $\alpha$-node $\alpha_i$, there exist a $\beta$-node $\beta_j$ and a path from $\beta_j$ to $\alpha_i$. 
For a given subset $V'$ of $V$, we let $N_{\rm in}(V')$ be the set of in-neighbors of $V'$, 
$$
N_{\rm in}(V'):= \{ v_i\in V \mid \exists v_j\in V' \mbox{ s.t. } (v_i,v_j)\in \E \}. 
$$
We gather the following result~\cite[Theorem 1]{olshevsky2015minimum}: 

\begin{theorem}\label{thm:sin}
A sparsity pattern is structurally controllable if and only if (1) $\G$ is accessible to $\beta$-nodes, and (2) for any subset $V'\subseteq V_\alpha$, $|N_{\rm in}(V')| \ge |V'|$.   	
\end{theorem}

\subsection{Structural ensemble controllability}

We now return to the ensemble setting~\eqref{eq:ensemble}, where $A$ and $B$ are matrix-valued functions. For a  sparsity pattern $\G$, the vector space $\V(\G)$ is a subspace of $\mathrm{C}^0(\Sigma,\real^{n\times n} \times \real^{n\times m})$.  

Following~\cite{XC:20}, we say that a sparsity pattern $\G$ is {\em structurally ensemble controllable} if there exists a pair $(A, B)\in \V(\G)$ such that the ensemble system~\eqref{eq:ensemble} is $\mathrm{L}^p$-controllable for some (and, hence, any, see~\cite{XC:20})  $p = 1,\ldots, \infty$.

We now recall the main result of~\cite{XC:20}.
%characterizing sparsity patterns that are structural ensemble controllability.  
To state this result, recall that $ \G=(V,\E) $  is said to admit a Hamiltonian decomposition if it contains a subgraph $ \G'=(V,\E') $, where $ \E'\subseteq \E $ such that $ \G' $ is a disjoint union of cycles.

\begin{theorem}\label{thm:XC} A sparsity pattern $\G$ is structurally ensemble controllable if and only if (1) $\G$ is accessible to the $\beta$-nodes, and (2) the subgraph of $\G$ induced by $V_\alpha$ admits a Hamiltonian decomposition.
\end{theorem}

\begin{remark}\label{rmk:ensemblestronger}
Note that condition (2) in Theorem~\ref{thm:XC} is stronger than condition (2) in Theorem~\ref{thm:sin} (details can be found in~\cite{XC:20}). Thus, if $\G$ is structurally ensemble controllable, then it is structurally controllable, but not vice versa.   
\end{remark}

\subsection{Structural averaged controllability}

We now arrive at the central part of what we will study in this paper, namely, structural averaged controllability. 
We still consider the ensemble setting~\eqref{eq:ensemble} so that $A$ and $B$ are matrix-valued functions, and $\V(\G)$ is a subspace of $\mathrm{C}^0(\Sigma,\real^{n\times n} \times \real^{n\times m})$ as was introduced earlier. 

%We have the following definition: 

\begin{definition}\label{def:average-structural-controllable}
A sparsity pattern $\G$ is {\bf structurally averaged controllable} if there exists a pair $ (A,B) \in \V(\G) $ such that the resulting system~\eqref{eq:ensemble} is averaged controllable
%\footnote{\color{blue} A slight variation of the definition, which does not affect our main results, is to relax the continuity of $(A,B)$ and require  that $A$ and $B$ are measurable and uniformly bounded (as is assumed in~\cite{EZ:14}).} 
%We note here that under such variation, all the results established in Section~\ref{section:main}, as well as the proofs, still hold.}. 
\end{definition}

 One can relax the continuity of $(A,B)$ and require that $A$ and $B$ are measurable and uniformly bounded (as is assumed in~\cite{EZ:14}); our main results to follow still apply in this case.
In the sequel, we compare the notion of structural averaged controllability with the other two; we show that it is the weakest among all the three notions. Moreover, we present a sparsity pattern that is structurally averaged controllable, but not structurally (ensemble) controllable. This is done in Section~\ref{subsection:comp}. Then, in Sections~\ref{subsection:necessary} and~\ref{subsection:sufficient}, respectively, we present a few novel necessary or sufficient conditions for sparsity patterns to be structurally averaged controllable. The Appendix includes a novel problem on variations of Hilbert matrices, which may be of independent interest.

\section{Main Results and Proofs}\label{section:main}
We now present and establish the main results of the paper. 

\subsection{Comparison between different structural controllability}\label{subsection:comp}

We start by establishing the following result:

\begin{proposition}\label{prop:weakest}
	If $\G$ is structurally (ensemble) controllable, then $\G$ is also structurally averaged controllable. 
\end{proposition}

\begin{proof}
We exhibit continuous functions $A:\Sigma \to \real^{n\times n}$ and $B: \Sigma\to \real^{n\times m}$ such that the resulting ensemble system~\eqref{eq:ensemble} is averaged controllable. 
Since $\G$ is structurally controllable,	 we let $(A', B')\in \real^{n\times n} \times \real^{n\times m}$ be compliant with $\G$ such that the single linear control system~\eqref{eq:singlelinearsystem} is controllable. 
Thus, the columns of $\{A'^kB'\}^{n-1}_{k = 0}$ span $\real^n$. 
Now, let $A(\sigma):= A'$ and $B(\sigma):= B'$ for all $\sigma\in \Sigma$, i.e., $A$ and $B$ are constant functions.  
Then, 
$$
\int_0^1 A^k(\sigma) B(\sigma) \mathrm{d} \sigma = A'^k B', 
$$  
for any $k \ge 0$. It then follows that the span of the columns of $\{\int_0^1 A^k(\sigma) B(\sigma) \mathrm{d} \sigma \}_{k\ge 0}$ is the span of the columns of $\{A'^k B'\}_{k \ge 0}$, which is $\real^n$. 
\end{proof}

We next show that the converse of Proposition~\ref{prop:weakest} is not true, i.e., $\G$ being structurally averaged controllable does not imply that $\G$ is structural (ensemble) controllable.  

\begin{proposition}\label{prop:nontrivialexample}
There exists a sparsity pattern that is structurally averaged controllable, yet not structurally (ensemble) controllable. 
\end{proposition}

\begin{proof}
Consider a sparsity pattern $\G = (V, \E)$ on three $\alpha$-nodes $\alpha_1$, $\alpha_2$, $\alpha_3$, and one $\beta$-node $\beta_1$. 
The edge set $\E$ is given by 
$
\E = \{(\alpha_1, \alpha_i), (\beta_1,\alpha_1) \mid i = 1,2,3\}
$. 
Thus, any pair $(A, B)$ compliant with $\G$ takes the following form: 
\begin{align*}
A=\begin{bmatrix}
\star & 0 & 0 \\
\star & 0 & 0\\
\star & 0 & 0\\
\end{bmatrix}
\quad \mathrm{and} \quad 
B=\begin{bmatrix}
\star\\
0  \\
0
\end{bmatrix}.
\end{align*}
Note that the in-neighbor of $V_\alpha$ is given by
$
N_{\rm in}(V_\alpha) = \{\alpha_1, \beta_1\}
$, 
so $|N_{\rm in}(V_\alpha)| < |V_\alpha|$. 
Thus, $\G$ is not structurally controllable by Theorem~\ref{thm:sin} nor structural ensemble controllable by Remark~\ref{rmk:ensemblestronger}.   

We now show that $\G$ is structurally averaged controllable. Choose the following pair $(A, B)$ in $\V(\G)$:  
\begin{align*}
A(\sigma):=
\begin{bmatrix}
\sigma & 0 & 0 \\
\sigma^2 & 0 & 0\\
\sigma^3 & 0 & 0\\
\end{bmatrix}
\quad \mbox{and} \quad 
B(\sigma):=
\begin{bmatrix}
1\\
0\\
0
\end{bmatrix},
\end{align*}
for all $\sigma\in \Sigma$. 
Then, by computation,  
\begin{align*}
 \begin{bmatrix}
\int^1_0 B(\sigma) \mathrm{d}\sigma & \int_0^1A(\sigma)B(\sigma)\mathrm{d}\sigma & \int_0^1A^2(\sigma)B(\sigma)\mathrm{d}\sigma
\end{bmatrix} 
= 
\begin{bmatrix}
1 & 1/2 & 1/3\\
0 & 1/3 & 1/4\\
0 & 1/4 & 1/5\\
\end{bmatrix},
\end{align*}
which is of full rank, rendering the resulting system~\eqref{eq:ensemble} averaged controllable by Theorem~\ref{theorem:averaged}. 
This shows that $\G$ is structurally averaged controllable. 
\end{proof}

Proposition~\ref{prop:nontrivialexample} shows that a necessary and sufficient condition for a sparsity pattern $\G$ to be structurally averaged controllable is different from those existing ones about structural (ensemble) controllability (e.g., Theorems~\ref{thm:sin} and~\ref{thm:XC}). 
Having a complete characterization of any such condition appears to be hard. We provide below a few necessary {\em or} sufficient conditions. 

\subsection{On necessary conditions}\label{subsection:necessary}
In this section, we present a few necessary conditions for a sparsity pattern $\G$ to be structurally averaged controllable. We first have the following result:

\begin{lemma}\label{lem:casek0}
	If $\G$ is structurally averaged controllable, then $\G$ is accessible to $\beta$-nodes. 
\end{lemma}

\begin{proof}
Suppose that $\G$ is not accessible to $\beta$-nodes; then, by relabelling (if necessary) the $\alpha$-nodes, we have that every pair $(A, B)\in \V(\G)$ takes the following form (the so-called ``Form I'' in~\cite{CTL:74}):   	
$$
A =
\begin{bmatrix}
	A_{11} & 0\\ 
	A_{21} & A_{22}
\end{bmatrix} \quad \mbox{and} \quad 
B = 
\begin{bmatrix}
	0 \\
	B_2
\end{bmatrix},
$$
where $A_{11}$ is $r\times r$ for $r \ge 1$, $A_{22}$ is $(n-r)\times (n-r)$, and $B_2$ is $(n -r)\times m$. It should be clear that for any $k\ge 0$, 
$$
\int^1_0 A^k(\sigma) B(\sigma) \mathrm{d}\sigma = 
\begin{bmatrix}
	0 \\
	\int^1_0 A^k_{22}(\sigma) B_{2}(\sigma) \mathrm{d} \sigma
\end{bmatrix}. 
$$  
Thus, $\{\int^1_0A^k(\sigma)B(\sigma) \mathrm{d}\sigma\}^\infty_{k = 0}$ cannot span $\real^n$. 
\end{proof}

The above result shows that condition (1) in Theorems~\ref{thm:sin} and~\ref{thm:XC} is also necessary for $\G$ to be structurally averaged controllable. 

Lemma~\ref{lem:casek0} can be generalized as follows. For a nonnegative integer $k$, let $U_\alpha(k)$ be the set of $\alpha$-nodes $\alpha_i$ such that there does {\em not} exist a walk from any $\beta$-node to $\alpha_i$ of length greater than $k$. For the case $k=0$, if $U_\alpha(0)$ is non-empty, then $\G$ is not accessible to the $\beta$-nodes. For a general $k$, we have the following result:

\begin{proposition}\label{prop:generalk}
Let $m$ be the number of $\beta$-nodes in $\G$. 
If $\G$ is structurally averaged controllable, then there does {\em not} exist a nonnegative integer $k$ such that $|U_\alpha(k)|> mk$.  
\end{proposition}

\begin{proof}	
We assume that such $k \ge 0$ exists and show that $\G$ cannot be structurally averaged controllable. 

	If $k = 0$, then $|U_\alpha(0)| > 0$, i.e., $U_\alpha(0)$ is nonempty, so $\G$ is not accessible to $\beta$-nodes. By Lemma~\ref{lem:casek0}, $\G$ is not structurally averaged controllable. We thus assume for the remainder of the proof that $k$ is strictly positive and $\G$ is accessible to the $\beta$-nodes.

   For any $(A, B)\in \V(\G)$, we show that the resulting ensemble system~\eqref{eq:ensemble} is not averaged controllable. 
	Consider the $i$th row of $A^k$, for any $\alpha_i\in U_\alpha(k)$. We claim that the row is identically zero. Suppose not, say the $ij$the entry of $A^k$ is not identically zero; then there exists a walk of length $k$ in $\G$ from $\alpha_j$ to $\alpha_i$. By assumption, $\G$ is accessible to the $\beta$-nodes, there exists a path from a $\beta$-node to $\alpha_j$. Concatenating the path with the walk, we obtain a walk from the $\beta$-node to $\alpha_i$ of length greater than $k$, which is a contradiction. 
	
	It then follows that for any $\alpha_i\in U_\alpha(k)$ and for any $\ell \ge k$, the $i$th row of $A^\ell B$ is identically zero. Thus, the span of the columns of $\{\int_0^1 A^\ell(\sigma) B(\sigma)\mathrm{d}\sigma\}^\infty_{\ell = k}$ has dimension at most $(n - |U_\alpha(k)|)$.      
	
	However, there are only $mk$ column vectors in $\{\int_0^1 A^\ell(\sigma) B(\sigma)\mathrm{d}\sigma\}^{k-1}_{\ell = 0}$, so the dimension of their span is at most $mk$. Thus, the span of $\{\int_0^1 A^\ell(\sigma) B(\sigma)\mathrm{d}\sigma\}^\infty_{\ell = 0}$ has dimension at most $(n - |U_\alpha(k)|) + mk < n$. By Theorem~\ref{theorem:averaged}, the ensemble system~\eqref{eq:ensemble} cannot be averaged controllable. 
\end{proof}

We have the following example to illustrate the  necessary condition in Proposition~\ref{prop:generalk}. 

\begin{example}\label{ex:cycle_not_sufficient}
Consider the graph $ \G $ shown in Fig.~\ref{fig:cycle_not_sufficient}. There is a single $\beta$-node and, hence, $m = 1$. 
The graph is accessible to the  $\beta$-node. Thus, it satisfies the necessary condition stated in Lemma~\ref{lem:casek0}. However, it does not satisfy the necessary condition stated in Proposition~\ref{prop:generalk}. Specifically, we have that 
 $$U_\alpha(k) = 
 \begin{cases}
 	 \varnothing & \mbox{if } k = 0, \\
 	 \{\alpha_2\} & \mbox{if } k = 1, \\
 	 \{\alpha_2,\alpha_3,\alpha_4\} & \mbox{if } k\ge 2.
 \end{cases}
$$
In particular, for $k = 2$, we have that $|U_\alpha(2)| > 2$. Thus, by Propositon~\ref{prop:generalk}, $\G$ is {\em not} structurally averaged controllable. 
\end{example} 
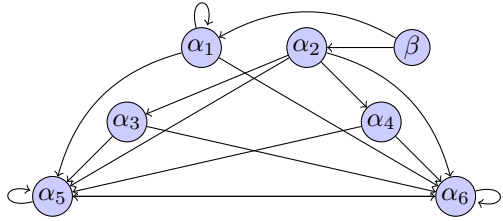
\begin{figure}[htb!]
\begin{center}
\begin{tikzpicture}[scale=0.9,node distance=1.4cm]
    \node[main node] (1) {$\alpha_1$};
    \node[main node] (2) [right of=1]  {$\alpha_2$};
                \node[main node] (7) [right of=2]  {$\beta$};
    \node[main node] (3) [below left of= 1] {$\alpha_3$};
    \node[main node] (4) [below right of= 2] {$\alpha_4$};
    \node[main node] (5) [below left of= 3] {$\alpha_5$}; 
        \node[main node] (6) [below right of= 4] {$\alpha_6$}; 
    \path[draw,->]
    (7) edge[bend right]  (1)
    (7) edge  (2)
    (1) edge[loop above] (1)
    (2) edge  (3)
    (2) edge  (4)
    (1) edge[bend right]  (5)
    (2) edge  (5)    
    (3) edge  (5)
    (4) edge  (5)
    (5) edge[loop left] (5)
    (6) edge  (5)
    (1) edge  (6)    
    (2) edge[bend left]  (6)    
    (3) edge  (6)        
    (4) edge  (6)    
    (5) edge  (6)    
    (6) edge[loop right] (6);        
    \end{tikzpicture}
\end{center}
\caption{Graph considered in  Example~\ref{ex:cycle_not_sufficient}.}\label{fig:cycle_not_sufficient}
\end{figure}

The above example also shows that adding more cycles or self-loops in the subgraph induced by the $\alpha$-nodes is in favor of the necessary condition stated in Proposition~\ref{prop:generalk}. Indeed, if we add a self-loop to node $\alpha_2$ in the graph shown in Fig.~\ref{fig:cycle_not_sufficient}, then $U_\alpha(k) = \varnothing$ for all $k \ge 0$.  

The importance of having a cycle or a self-loop is even more significant when $\G$ does {\em not} satisfy condition (2) in Theorem~\ref{thm:sin}, i.e., there exists a subset $V' \subseteq V_\alpha$ such that $|N_{\rm in}(V')| \ge |V'|$. For instance, in the example given in the proof of Proposition~\ref{prop:nontrivialexample}, if we remove the self-loop on node $\alpha_1$, then $A^k B\equiv 0$ for any $k\ge 2$ and for any $(A, B)\in \V(\G)$. Thus, the span of the columns $\{\int_0^1 A^k(\sigma) B(\sigma)\mathrm{d}\sigma\}^\infty_{k = 0}$ has dimension at most $2$, so the modified $\G$ is not structurally averaged controllable.  
 To further formalize the observation, we have the following result as a corollary to Proposition~\ref{prop:generalk}:
 
\begin{corollary}\label{cor:cycleforValpha}
If $\G$ is structurally averaged controllable, with a single $\beta$-node, and if $|N_{\rm in}(V_\alpha)| < |V_\alpha| = n$, then  $\G$  cannot be acyclic.  	
\end{corollary}

\begin{proof} 
We assume that $G$ is acyclic and show that $G$ is not structurally averaged controllable. 
First, by definition of $U_\alpha(k)$, we have the  inclusion sequence: 
$U_\alpha(0)\subseteq U_\alpha(1) \subseteq \cdots \subseteq U_\alpha(n)$. 
Since $\G$ is acyclic and since there are $n$ $\alpha$-nodes and one single $\beta$-node, there exists an integer $k \le n$ such that $U_\alpha(k)= V_\alpha$. We show below that $k$ can be chosen to be $(n-1)$.  
Suppose not, i.e., 
$U_\alpha(n-1) \subsetneq U_\alpha(n) = V_\alpha$;
then, there exists a node $\alpha_i$ and a walk from the $\beta$-node to $\alpha_i$ of length $n$. Because $\G$ is acyclic, the walk has to be a {\em path}, i.e., all nodes appearing in the walk are distinct. 
More specifically, if we express the path explicitly as $\beta, \alpha_{j_1}\ldots \alpha_{j_n}$ with $\alpha_{j_n} = \alpha_i$, then the nodes $\alpha_{i_\ell}$, for $\ell = 1,\ldots, n$, are the $n$ distinct nodes in $V_\alpha$. But, this implies that $N_{\rm in}(V_\alpha)$ includes at least $n$ nodes, namely, $\alpha_{j_1},\ldots, \alpha_{j_{n-1}}$ and $\beta$, which contradicts the assumption that $|N_{\rm in}(V_\alpha)| < |V_\alpha| = n$. 
Now, we have that $|U_\alpha(n-1)| = n > (n-1)$. Thus, by Proposition~\ref{prop:generalk}, $\G$ is not structurally averaged controllable, concluding the proof. 
\end{proof}

The above corollary can slightly be generalized as follows (we omit the proof because the arguments are similar to the ones in the above proof): 

\begin{corollary}\label{corollary:cycle-necessary}
Let $\G$ be structurally averaged controllable. 
Suppose that there exist a subset $V'_\alpha\subseteq V_\alpha$ and a $\beta$-node, say $\beta_i$, such that 
\begin{equation}\label{eq:subset-cycle}
N_{\rm in}(V'_\alpha) \subseteq V'_\alpha \cup \{\beta_i\} \quad \mbox{and} \quad |N_{\rm in}(V'_\alpha)| < |V'_\alpha|;
\end{equation}
then, the subgraph of $\G$ induced by $V'_\alpha$    is not acyclic. 
\end{corollary}

\subsection{Toward sufficient conditions}\label{subsection:sufficient}

In this section, we investigate sufficient conditions for structural averaged controllability. Throughout, we focus on the case where $ m=1 $. 
Corollary~\ref{cor:cycleforValpha} demonstrates the need for existence of a cycle in $ V_\alpha$ whenever $|N_{\rm in}(V_\alpha)|< |V_\alpha|$.
A simple case we will focus on is that this cycle is a self-loop. More specifically, we have the following assumption:

\begin{assumption}\label{assumption:1-cycle}
There is only one single $\beta$-node. There exists a node $\alpha_1\in V_\alpha$ such that $\alpha_1$ has a self-loop and $(\beta,\alpha_1)$ is an edge of $\G$. Moreover, the subgraph $\G_\alpha$ induced by $V_\alpha$ contains a directed spanning tree $\mathcal{T}_\alpha$ with $\alpha_1$ the root.  
\end{assumption}

The sparsity pattern given in the proof of Proposition~\ref{prop:nontrivialexample} is a typical example that satisfies Assumption~\ref{assumption:1-cycle}.  
The next result generalizes the pattern: 

\begin{theorem}\label{theorem:rooted}
Suppose that $\G$ satisfies Assumption~\ref{assumption:1-cycle} and that there is an edge from $\alpha_1$ to any other $\alpha$-node; then $\G $ is structurally averaged controllable.
\end{theorem}

We postpone the proof of this result, as it will follow as a corollary of Theorem~\ref{theorem:k-cycle} below, which we focus on next. We first gather some useful terminologies.  
%The sequence of results below identify some minimal structures for structural averaged controllability. 

\subsubsection{A partition for $ V_\alpha $}
Let $\T_\alpha$ be the directed spanning tree of $\G_\alpha$ as introduced in Assumption~\ref{assumption:1-cycle}.  
For any $\alpha_i$, there exists a {\em unique} path from the root $\alpha_1$ to $\alpha_i$ within $\T_\alpha$. We let the depth of $\alpha_i$ be the length of the path. Denote by $V_\alpha(k)$, for $k\ge 0$, the subset of $V_\alpha$ that is composed of all nodes of depth~$k$. For $k = 0$, the set $V_\alpha(0)$ is a singleton $\{\alpha_1\}$. 
A partition of $V_\alpha$ based on the depths of nodes can be obtained as follows:  
\begin{equation}\label{eq:partition}
V_\alpha=\cup_{k=0}^p V_\alpha(k), 
\end{equation}
where $p$ is the maximum depth. We will use such a partition later for constructing an averaged controllable $(A, B)$ pair. 

\subsubsection{Variations of Hilbert matrices}
We start by recalling some preliminaries on Hilbert matrices, which we utilize in the statement of our main result. 
Let $ \N $ denote the set of non-negative integers and let $ n \in \N_{\geq 2} $. A Hilbert matrix $ H_n $ is an $ n \times n $ matrix whose entries are 
\[
(H_n)_{ij}=\frac{1}{i+j-1}. 
\]
Hilbert matrices are invertible, and in fact positive definite, and their inverse can explicitly be computed; in particular, 
%\begin{equation}\label{eq:H_inv}
\begin{align*}
(H^{-1}_n)_{ij}=(-1)^{i+j}&(i+j-1)\binom{n+i-1}{n-j}\\
&\times \binom{n+j-1}{n-i}\binom{i+j-2}{i-1}^2.
\end{align*}
%\end{equation}

Let $ e_i \in \real^n $ be the unit vector with the $ i $th entry equal one. For a $ \gamma \in \{1,\ldots, n\} $, 
let $ \{\ell_i\}_{i=1}^{\gamma} $ be a \emph{non-decreasing} sequence of positive integers such that $ i\le \ell_i < n $ for all $ i $. Let $ u^{(\ell_{i})} \in \real^n $ be defined through its components as 
\[
u^{(\ell_{i})}_j:=
\begin{cases}
0 & j\leq \ell_{i},\\
\frac{1}{i+j-1} & j> \ell_{i}.\\
\end{cases}
\]
Now, for the given sequence $\{\ell_i\}^\gamma_{i = 1}$,  we let 
\begin{equation}\label{eq:zero-Hilbert}
H_n(\ell_1,\cdots, \ell_{\gamma}):=H_n-\sum_{i=1}^{\gamma}u^{(\ell_i)} e^\top_i.
\end{equation}
Note that $ H_n(\ell_1,\cdots, \ell_{\gamma}) $ is simply obtained by setting entries in column $ i $ with row larger than $ \ell_i $ in $ H_n $ to zero. For example, for $ n=5 $ and $ \gamma=3 $ we have that 
\begin{equation}\label{eq:example-H}
H_n(2,2,3)=
\begin{bmatrix}
1&1/2&1/3&1/4&1/5\\
1/2&1/3&1/4&1/5&1/6\\
 0&  0&1/5&1/6&1/7\\
 0&  0&  0&1/7&1/8\\
 0&  0&  0&1/8&1/9
\end{bmatrix}.
\end{equation}
For a special case, we prove in Corollary~\ref{corollary:Hilbert-zero} in the Appendix that $ H_n(\ell_1,\cdots, \ell_{\gamma}) $ as defined above is invertible. In fact, we conjecture that  $ H_n(\ell_1,\cdots, \ell_{\gamma}) $ is in general invertible; this conjecture appears to be difficult to establish. For the purpose of our work, we make this as an assumption:
\begin{assumption}\label{assumption:LA}
Suppose that $ V_\alpha $ is partitioned as in~\eqref{eq:partition}. We assume that $ H_n(\ell_1,\cdots, \ell_{p}) $, with~$ \ell_j $ given as
\begin{equation}\label{eq:ell}
\ell_j:=1+\sum_{k=1}^{j-1} |V_\alpha(k)|,
\end{equation}
is invertible. 
\end{assumption}

Note that in~\eqref {eq:ell}, we have that $ \ell_j \geq j $. This holds because the cardinality of $ V_\alpha(k) $ is at least one, for $ 1\leq k \leq j-1 $.

\subsubsection{A constructive sufficiency result}

With the preliminaries above, 
we are now in a position  to state the main result of this section:

\begin{theorem}\label{theorem:k-cycle}
Suppose that $ \G $ satisfies Assumptions~\ref{assumption:1-cycle} and~\ref{assumption:LA}; then, $ \G $ is structurally averaged controllable.
\end{theorem}

\begin{proof}
The proof is constructive. 
We first describe the process of assigning scalar functions to the entries of the pair $ (A,B) \in \V(\G) $. After that, we show that the constructed pair is structurally averaged controllable. We still let $\T_\alpha $ be the directed spanning tree of $\G_\alpha $ as introduced in Assumption~\ref{assumption:1-cycle}. 
We specify below the entries of $A$ and $B$ that correspond to the edge $(\beta, \alpha_1)$ and the edges in $\T_\alpha$. All the other entries of $A$ and $B$ are set to be identically zero.  

The edge $(\beta, \alpha_1)$ corresponds to the first entry $b_1$ of $B$ (note that $B$ is a column since $m = 1$). We set $b_1\equiv 1$. 

The edges in $\T_\alpha$ correspond to the entires in $A$. We re-label, if necessary, the $\alpha$-nodes so that the sub-indices of nodes in $ V_\alpha(k)$ are larger than those in $ V_\alpha(k')$, whenever $ k>k' $. 
%We also let the $ \beta $-node be node $ 1 $. 
To this end, we perform the following assignments: 
\begin{itemize}
\item For the self-loop $ (\alpha_1,\alpha_1)$, we let $ a_{11}\in \mathrm{C}^{0}(\Sigma,\real) $ be given by $ a_{11}: \sigma\mapsto \sigma $; 
\item For each edge $ (\alpha_1,\alpha_i) $ with $ \alpha_i \in V_\alpha(1) $,
we let $ a_{i1}\in \mathrm{C}^{0}(\Sigma,\real) $ be given by $ a_{i1}: \sigma\mapsto \sigma^{i} $; 
\item For each edge $ (\alpha_i,\alpha_j) $ with $ \alpha_i \in V_\alpha(k) $ and $ \alpha_j \in V_\alpha(k+1)$, where $ 1\leq k\leq p$, we let $ a_{ij}\in \mathrm{C}^{0}(\Sigma,\real) $ be given by $ a_{ij}: \sigma\mapsto \sigma^{j-i+1} $. 
\end{itemize}
An example of such assignment is given by Fig.~\ref{fig:main6_cont}. 
We then define a matrix-valued function as follows: 
\begin{align*}
\mathcal{C}:=[B\,\, AB\, 
\cdots \, A^{n-1}B].
\end{align*}
We now claim that for the $ (A,B) \in \V(\G) $ constructed above and for the sequence $ \{\ell_j\}_{j=1}^{p} $ of positive integers with $ \ell_j $ defined in~\eqref{eq:ell}, with $ \ell_j \geq j $, the following equality is satisfied: 
\begin{equation}\label{eq:integralofC}
\int_0^1\mathcal{C}(\sigma)\mathrm{d}\sigma= H_n(\ell_1,\cdots, \ell_{p}).
\end{equation}
Note that, by Assumption~\ref{assumption:LA}, this will prove that $ \G $ is structurally averaged controllable. We now prove this claim.

We show that the for any $j = 1,\ldots, p$, the last $(n - \ell_j)$ entries of the $j$th column of $\mathcal{C}$ are identically zero, where $\ell_j$ is defined in~\eqref{eq:ell}. 
The first column of $\mathcal{C}$ is simply the vector $B$. By construction, only the first entry~$b_1$ of $B$ is not identically zero. 
Next, let $\T$ be the subgraph of $\G$ obtained by adding the $\beta$-node and the edge $(\beta,\alpha_1)$ to $\T_\alpha$. The $i$th entry of $A^{j-1}B$ (i.e., the $j$th column of $\mathcal{C}$), for $j \in \{2,\ldots,p\}$, 
is not identically zero if and only if there is a path from $\beta$ to $\alpha_i$ in $\T$ of length less than or equal to~$j$ (or, equivalently, there is a path from $\alpha_1$ to $\alpha_i$ in $\T_\alpha$ of length less than~$j$).  
By definition of $V_\alpha(k)$, if $\alpha_i\in V_\alpha(k)$ for $k \ge j$, there does not exist a path from $\alpha_1$ to $\alpha_i$ of length less than~$j$ and, hence, the $i$th entry of the $j$th column in $\mathcal{C}$ is identically zero. 
Thus, by the ordering chosen for the $\alpha$-nodes, the last $(n - \ell_j)$  entries of column $j$ in $ \mathcal{C}$, for $ j \in \{2,\ldots, p\} $, are identically zero. 

For the other entries of $\mathcal{C}$ i.e., entries $\mathcal{C}_{ij}$ with the depth of $\alpha_i$ strictly less than $j$, 
we use again the relationship between powers of $A$ and the lengths of paths from node $\alpha_1$, and obtain that 
\begin{equation*}\label{eq:Cij}
\mathcal{C}_{ij} = 
a_{11}^{(j-1)-i^\star} a_{i_1,1} \cdots a_{i,i_{i^\star-1}}
\end{equation*}
where  $ i^\star$ is the depth of $\alpha_i$ and $(\alpha_i,\alpha_{i_{i^\star-1}},\cdots,\alpha_{i_1},\alpha_1)$ is the {\em unique path} from $\alpha_1$ to $\alpha_i $ in~$\T_\alpha$. 

Now, with the choice of the nonzero entries described in the above items, we have that 
\begin{align*}\label{eq:Cij-replaced}
\mathcal{C}_{ij} (\sigma)&
= \begin{cases}
\sigma^{(j-1)-i^\star}\sigma^{i_1} \cdots 
\sigma^{i-i_{i^\star-1}+1}
& i^*<j \\
0  & i^*\geq j 
\end{cases}\nonumber\\
&= \begin{cases}
\sigma^{(i+j-2)}
& i^*<j \\
0  & i^*\geq j. 
\end{cases}
\end{align*}
We thus conclude that~\eqref{eq:integralofC} holds. 
\end{proof}

We remark here that the graph-theoretic construction of the controllability matrix in our proof is universal, in that any other assignment would still lead to study of a matrix with the same sparsity pattern as the one obtained for the class of sparse Hilbert matrices. The result of Theorem~\ref{theorem:k-cycle} relies on the invertibility Assumption~\ref{assumption:LA}. The highlight of this result is that it reduces the problem of structural averaged controllability to a completely linear algebraic conjecture on invertibility of a class of sparse Hilbert matrices.  Theorem~\ref{thm:1} in the Appendix addresses a specific scenario where we know this assumption can be removed, and indeed the proof of Theorem~\ref{theorem:rooted} can be deduced using this.

\quad {\em Proof of Theorem~\ref{theorem:rooted}:}
According to the procedure proposed in Theorem~\ref{theorem:k-cycle}, we let $ a_{i1}: \sigma \mapsto \sigma^{i} $, and  $b_1\equiv 1 $. Clearly, $ \ell_1=1 $. Since all other nodes have $ \alpha_1 $ as the in-neighbor in $\T_\alpha$, $ V_\alpha(1) $ contains all the remaining nodes and, hence, $ \ell_2=n $, which by our convention in~\eqref{eq:zero-Hilbert}, means that 
\[
\int_0^1\mathcal{C}(\sigma)\mathrm{d}\sigma= H_n(\ell_1), 
\]
and it is invertible by Theorem~\ref{thm:1}.
\hfill{\subqed}

This result does not fully utilize Theorem~\ref{thm:1}, however, as one could verify this invertibility by using the fact that the first column of $H_n(\ell_1)$ has only one nonzero entry and, hence, the determinant is nonzero as the principal minor obtained by removing the first column/row is a Cauchy matrix itself. This being said, one can generate further instances where Assumption~\ref{assumption:1-cycle} can be removed by considering Corollary~\ref{corollary:Hilbert-zero}, which ensures invertibility of matrices such as the one displayed in~\eqref{eq:example-H}. We end the section with an illustration. 

\begin{example}\label{ex:main6_cont}
Consider the graph in Fig.~\ref{fig:main6_cont}.
The graph has a self-loop on node $\alpha_1$ shown in red. The subgraph $\G_\alpha$ is itself a directed tree. Therefore, 
$
V=\cup^4_{k = 0} V_\alpha(k)
$,  
where 
\begin{align*}
&V_\alpha(0)=\{\alpha_1\}, \
V_\alpha(1)=\{\alpha_2\}, \
V_\alpha(2)=\{\alpha_3,\alpha_4\}, \\
&V_\alpha(3)=\{\alpha_5\},  \
V_\alpha(4)=\{\alpha_6\}. %\quad \quad \oprocend
\end{align*}
We have displayed the assignment given in the proof of Theorem~\ref{theorem:k-cycle} for the pair $(A, B)\in \V(\G)$. The corresponding nonzero entries are also illustrated in Fig.~\ref{fig:main6_cont}. 
\begin{figure}[tb!]
\begin{center}
\begin{tikzpicture}[scale=1.0,node distance=1.4cm]
    \node[main node] (1) {$\alpha_1$};
    \node[main node] (2) [right of=1]  {$\alpha_2$};
                \node[main node] (7) [right of=2]  {$\beta$};
    \node[main node] (3) [below left of= 1] {$\alpha_3$};
    \node[main node] (4) [below of= 2] {$\alpha_4$};
    \node[main node] (5) [below of= 3] {$\alpha_5$}; 
        \node[main node] (6) [below left of= 4] {$\alpha_6$}; 
    \path[draw,->]
    (7) edge [dashed] [bend right] node[above] {\footnotesize$1$} (1)
    (1) edge [red] [loop above] node[above] {\footnotesize$\sigma$}  (1)
    (1) edge node[above right] {\footnotesize$\sigma^2$} (2)
    (2) edge node[below] {\footnotesize$\sigma^2$} (3)
    (2) edge node[right] {\footnotesize$\sigma^3$} (4)
    (3) edge node[left] {\footnotesize$\sigma^3$} (5)
    (5) edge node[above] {\footnotesize$\sigma^2$} (6);    
    \end{tikzpicture} 
\end{center}
\caption{Graph considered in Example~\ref{ex:main6_cont} and the assignment of scalar functions to the entries corresponding to the edges.}\label{fig:main6_cont}
\end{figure}
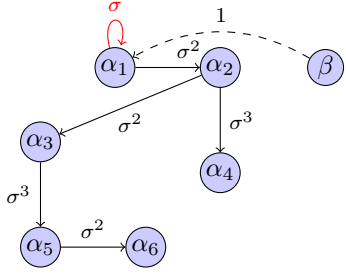
We have that 
$
\int_0^1\mathcal{C}(\sigma)\mathrm{d}\sigma
%=&
%\begin{bmatrix}
%1&1/2&1/3&1/4&1/5 & 1/6\\
%0&1/3&1/4&1/5&1/6 & 1/7\\
% 0&  0&1/5&1/6&1/7 & 1/8\\
% 0&  0&  1/6&1/7&1/8 & 1/9\\
% 0&  0&  0&1/8&1/9 & 1/10\\
% 0&  0&  0&0&1/10 & 1/11
%\end{bmatrix}\\
=H_6(1,2,4,5).
$
It is easy to observe that $
(\ell_1,\ell_2,\ell_3,\ell_4)=(1,2,4,5) $
agrees with~\eqref{eq:ell}. This matrix is readily invertible and, hence, by Theorem~\ref{theorem:k-cycle}, the system is structurally averaged controllable. 
\end{example}

\section{Conclusions}
We introduce and address a novel problem of structural averaged controllability for linear ensemble systems. Some necessary or sufficient conditions for a sparsity pattern to be structurally averaged controllable are provided. Although the parameterization space is chosen to be the closed interval $\Sigma = [0,1]$, we believe that the results established hold for general continuum spaces. Future work include characterizing necessary {\em and} sufficient conditions for sparsity patterns to be structurally averaged controllable, studying the minimal controllability problem~\cite{olshevsky2015minimum}, and extending the results to, e.g., bilinear systems~\cite{AT-MAB-BG:18}.    

%When $ \Sigma $ is not merely an interval, for instance the case when $ \Sigma=\mathbb{S}_1 $, the unit circle, is more delicate, and a differential geometric framework is in the background. This same issue is present in the study of structural controllability for ensembles, see~\cite{XC:20}; we hope to address this in a future work. Extending the results to framework of bilinear systems~\cite{AT-MAB-BG:18} and studying the minimal controllability problem~\cite{AO:14} in this context will also be of interest. 
%\end{remark}

\bibliographystyle{ieeetr}%
\bibliography{alias,BG,bib}

\vspace{-0.3cm}

\newcounter{mycounter}
\renewcommand{\themycounter}{A.\arabic{mycounter}}
\newtheorem{propositionappendix}[mycounter]{Proposition}
\newtheorem{conjectureappendix}[mycounter]{Conjecture}
\newtheorem{theoremappendix}[mycounter]{Theorem}
\newtheorem{sublemmaappendix}[mycounter]{Sublemma}
\newtheorem{corollaryappendix}[mycounter]{Corollary}

\appendix

\section{Variations of Hilbert matrices with zero entries}\label{sec:appendix}

We have the following conjecture: 
\begin{conjectureappendix}
The matrix $ H_n(\ell_1,\cdots, \ell_{\alpha}) $ is invertible. 
\end{conjectureappendix}

We now prove a preliminary version of this conjecture for $ \alpha=1 $. For this purpose, let $ e_1 \in \real^n $ be the unit vector with first entry one and for $ \ell \in \N_{\geq 1} $, let $ u^{(\ell)} \in \real^n $ be defined through its $ j $th component: 
\[
u^{(\ell)}_j:=
\begin{cases}
0 & j\leq \ell,\\
\frac{1}{j} & j> \ell.\\
\end{cases}
\]
%Consider now the matrix
%$
%H_n(\ell):=H_n-u^{(\ell)} e^\top_1
%$. 
%We prove the following result. 
\begin{theoremappendix}\label{thm:1}
The matrix $ H_n(\ell):=H_n-u^{(\ell)} e^\top_1$ is invertible for any $ \ell \in \N_{\geq 1} $.
\end{theoremappendix}
\begin{proof}
By the matrix determinant lemma, we have
\begin{align*}
\det(H_n(\ell))
=&(1-e^\top_1 H_n^{-1}u^{(\ell)})\det(H_n) \\
=&(1-\sum_{k=1}^n (H_n^{-1})_{1k}u^{(\ell)}_k)\det(H_n)\\
=&(1-\sum_{k=1}^n (-1)^{(k+1)}k\binom{n}{n-k}\binom{n+k-1}{n-1}
u^{(\ell)}_k)\det(H_n)\\
=&\left (1-\sum_{k=\ell+1}^n (-1)^{(k+1)}\binom{n}{n-k}\binom{n+k-1}{n-1} \right )\det(H_n).
\end{align*}
Clearly, the result follows if we prove that for $ n \in \N_{\geq 2} $ and $ \ell \in \{1,\ldots, n-1\} $,
%\footnote{
%Related to this statement, I have the following conjecture, which appears to be more difficult to prove. 
%\begin{conjecture}
%For $ \ell \in \N_{\geq 0} $, $ p \in \N_{\geq 1}$, and $ n \in \N_{\geq 2}$ with $ \ell \leq p \leq n $, we have that
%\[
%\sum_{k=\ell+1}^p(-1)^{(k+1)}\binom{n}{n-k}\binom{n+k-1}{n+1}=1
%\]
%if and only if $ p=n $ and $ \ell=0 $.
%\end{conjecture}
%} 
\begin{equation}\label{eq:sum-not-1}
\sum_{k=\ell+1}^n (-1)^{(k+1)}\binom{n}{n-k}\binom{n+k-1}{n-1}\neq 1.
\end{equation}
To this end, using a backward induction on $ \ell $, one can show that
\[
\sum_{k=\ell+1}^n(-1)^{(k+1)}\binom{n}{n-k}\binom{n+k-1}{n-1}=\frac{1}{n^2}Z(\ell,n), 
\]
where 
\begin{equation}\label{eq:Z-form1}
Z(\ell,n):=(-1)^{\ell}(\ell+1)^2\binom{n}{n-\ell-1}\binom{n+\ell}{n-1}.
\end{equation}
Indeed. The statement can be seen to hold for $ \ell=n-1$. Suppose now by the way of induction that the statement holds for some $ \ell $, and let us prove that it also holds for $ \ell-1$. Note that 
\begin{align*}
n^2\sum_{k=\ell}^n(-1)^{(k+1)}&\binom{n}{n-k}\binom{n+k-1}{n-1}\\
=&n^2\sum_{k=\ell+1}^n(-1)^{(k+1)}\binom{n}{n-k}\binom{n+k-1}{n-1}\\
&+(-1)^{\ell+1}n^2\binom{n}{n-\ell}\binom{n+\ell-1}{n-1}\\
=&(-1)^{\ell}(\ell+1)^2\binom{n}{n-\ell-1}\binom{n+\ell}{n-1}\\
&+(-1)^{\ell+1}n^2\binom{n}{n-\ell}\binom{n+\ell-1}{n-1}\\
=&(-1)^{\ell-1}
\left(-(\ell+1)^2\times \frac{n-\ell}{\ell+1}\times \frac{n+\ell}{\ell+1}+n^2
\right)\binom{n}{n-\ell}\binom{n+\ell-1}{n-1}\\
=&(-1)^{\ell-1}\ell^2\binom{n}{n-\ell}\binom{n+\ell-1}{n-1}\\
=& Z(\ell-1,n),
\end{align*}
where we have used the induction assumption to arrive at the second equality. This finishes the induction argument. 

Using this result, in order to prove~\eqref{eq:sum-not-1}, it is enough to show that $ Z(\ell,n) \neq n^2 $, for  $ \ell \in \{1,\ldots, n-1\} $ (note that $ Z(0,n)=n^2 $). The next result is a stepping stone. 
\begin{sublemmaappendix}\label{sublemma-pick}
For $ n \in \N_{\geq 2}  $, we have that 
\[
\argmin_{\ell \in  \{1,\ldots,n-1\}} |Z(\ell,n)|=1,
\]
where $ Z(\cdot,n) $ is given by~\eqref{eq:Z-form1}.
\end{sublemmaappendix}
\begin{proof}
We first show that $ |Z(\cdot,n)| $ is increasing on $ \{1,\ldots, \ell^*\} $ and decreasing on $ \{\ell^*,\ldots, n-1\} $, where~$
\ell^*=~\lfloor\frac{n}{\sqrt{2}}\rfloor.
$
Note that for $ \ell \in  \{1,\ldots,n-1\} $, we have that
\[
Z(\ell,n)=-(\frac{n^2}{\ell^2}-1)Z(\ell-1,n), 
\]
and, hence 
\begin{equation*}
Z(\ell,n)-Z(\ell-1,n)=-\frac{n^2}{\ell^2}Z(\ell-1,n).
\end{equation*}
We also have 
\begin{equation*}
|Z(\ell,n)|-|Z(\ell-1,n)|=(\frac{n^2}{\ell^2}-2)|Z(\ell-1,n)|.
\end{equation*}
Hence, $ |Z(\ell,n)|>|Z(\ell-1,n)| $ if $ \ell\le \ell^* $ and $ |Z(\ell,n)|<|Z(\ell-1,n)| $, otherwise. As a result, the minimum value of $ |Z(\ell,n)| $ over $  \{1,\ldots,n-1\} $ occurs either at $ \ell=1 $ or at $\ell=n-1 $. We have that 
\[
|Z(1,n)|=n^4-n^2 \quad \mathrm{and} \quad |Z(n-1,n)|=n^2\binom{2n-1}{n-1}.
\]
Noting that $ |Z(1,n)| < |Z(n-1,n)| $ for all $ n \in \N_{\geq 3} $ yields the result. 
\end{proof}
By Sublemma~\ref{sublemma-pick}, $ \min_{\ell \in  \{1,\ldots,n-1\}} |Z(\ell,n)|=n^4-n^2 $, which cannot be equal to $ n^2 $ for any $ n \in \N_{\geq 2} $, proving that~\eqref{eq:sum-not-1} holds, yielding the claim.
\end{proof}

We now provide an extension of this result. Let $ H_n^{(j)} $ be a principal minor of the Hilbert matrix $ H_n $ of size $ r_j \times r_j $, where $ j \in \{1,\ldots, k\} $ and $ k \in \N_{\geq 1} $ and and let 
\begin{equation}\label{eq:Hjp}
H_n^{(j)}(p_j)= H_n^{(j)}-u^{(p_j)}e_1^\top,
\end{equation}
where $ p_j\geq 1 $. Then, by a proof similar to the one in Theorem~\ref{thm:1}, we have that $ H_n^{(j)}(p_j) $ is invertible. Using this, we state the following immediate extension of Theorem~\ref{thm:1}.

\begin{corollaryappendix}\label{corollary:Hilbert-zero}
If the matrix $ H_n(\ell_1,\cdots, \ell_{\gamma}) $ is of the form
\begin{align}\label{eq:inv-form}
H_n&(\ell_1,\cdots, \ell_{\gamma})=\cr
&\begin{bmatrix}
H_n^{(1)}(p_1) & \star & \cdots & \star\\
\mathbf{0}_{r_2\times r_1} & H_n^{(2)}(p_1) &  \cdots & \star\\
\mathbf{0}_{r_3\times r_1}  &  \mathbf{0}_{r3\times r_2}  &  \cdots & \star \\
\vdots & \vdots  & \ddots & \vdots &\\
\mathbf{0}_{r_k\times r_1}  &  \mathbf{0}_{r_k\times r_2}   &  \cdots  & H_n^{(k)}(p_k)  \\
\end{bmatrix},
\end{align}
where each $ H_n^{(j)} $ is defined in~\eqref{eq:Hjp}, $ j \in \{1,\ldots, k\} $, and $ k \in \N_{\geq 1} $, then
$ H_n(\ell_1,\cdots, \ell_{\gamma}) $ is invertible.
\end{corollaryappendix}
%Note that in~\eqref{eq:inv-form} the sequence of integer $ \{r_j\} $ can be related to the sequence $ \{\ell_i\}_{i=1}^{\gamma} $. The reader can verify that the example given by~\eqref{eq:example-H} is invertible according to this result. 

\end{document}